\documentclass[a4paper,11pt]{article}

\usepackage{latexsym,graphics,color}
\usepackage{amsmath,amsfonts,amssymb,amsthm}
\newtheorem{thm}{Theorem}[section]
\newtheorem{lem}[thm]{Lemma}

\newtheorem{cor}[thm]{Corollary}
\newtheorem{prop}[thm]{Proposition}

\newtheorem{rem}[thm]{Remark}

\input epsf

\title{Yet another Proof of an old Hat}
\author{Roland Bacher}

\begin{document}
\maketitle

\begin{abstract} Every odd prime number $p$ can be written in
  exactly $(p+1)/2$ ways as a sum $ab+cd$ with $\min(a,b)>\max(c,d)$
  of two ordered products. This gives a new proof Fermat's Theorem
  expressing primes
  of the form $1+4\mathbb N$ as sums of two squares 
  \footnote{
 Math. class: Primary:
 11A41.
 Secondary: 11H06.
}.
\end{abstract}


\begin{thm}\label{thmmain} For every odd prime number $p$ there exist
  $(p+1)/2$ ordered quadruplets $(a,b,c,d)$ in $\mathbb N$
  such that $p=a\cdot b+c\cdot d$ and $\min(a,b)>\max(c,d)$.
\end{thm}

As a consequence we obtain a new proof of the following result
discovered by an old rascal who did not want to spoil his margins
and left the proof to another chap who had no such qualms.

\begin{cor}\label{oldhat} Every prime number of the form $1+4\mathbb N$ is a sum
  of two squares.
\end{cor}

\begin{proof}[Proof of Corollary \ref{oldhat}] If $p$ is a prime-number congruent to
  $1\pmod 4$, the number $(p+1)/2$ of solutions $(a,b,c,d)$ defined by
  Theorem \ref{thmmain} is odd. The involution $(a,b,c,d)\longmapsto
  (b,a,d,c)$ has thus a fixed point $(a,a,c,c)$ expressing $p$ as a sum of two squares.
  \end{proof}

  Corollary \ref{oldhat} has of course already quite a few proofs.
  Some are described in the entry ``Fermat's theorem on sums of two squares'' of
  \cite{WikipFerm}.  
  The author enjoyed also the account given in \cite{Els}.

  The set $\mathcal S_p$ of solutions defined by Theorem \ref{thmmain}
  is invariant under the action of Klein's Vierergruppe with
  non-trivial elements acting by
  $$(a,b,c,d)\longmapsto (b,a,c,d),(a,b,d,c),(b,a,d,c)\ .$$
  The following tables list all elements $(a,b,c,d)$ with
  $a,b,c,d$ decreasing together with the size $\sharp(\mathcal O)$
  of the corresponding orbit under Klein's Vierergruppe
  for the sets
  $\mathcal S_{29}$ and $\mathcal S_{31}$:
  $$\begin{array}{cc}
      \begin{array}{cccc|r}
        a&b&c&d&\sharp(\mathcal O)\\
        \hline
29&1&0&0&2\\
14&2&1&1&2\\
7&4&1&1&2\\
9&3&2&1&4\\
5&5&4&1&2\\
5&5&2&2&1\\
5&4&3&3&2\\
\hline
&&&&15
      \end{array}\qquad&
       \begin{array}{cccc|r}
        a&b&c&d&\sharp(\mathcal O)\\
        \hline
31&1&0&0&2\\
15&2&1&1&2\\
10&3&1&1&2\\
6&5&1&1&2\\
7&4&3&1&4\\
9&3&2&2&2\\
5&5&3&2&2\\
\hline
&&&&16
       \end{array}
    \end{array}
    $$
    Establishing complete lists $\mathcal S_p$ of solutions
    for small primes is rather pleasant and rates among the author's more
    confessable procrastinations.
    
    We proceed now by giving an elementary proof of
    Theorem \ref{thmmain}.

    A last Section contains a few 
    remarks and ends with a somewhat speculative part.

\section{Proof of Theorem \ref{thmmain}}
  
We state the following reformulation of Pick's Theorem\footnote{Pick's theorem gives the area $\frac{1}{2}b+i-1$ of a closed lattice polygon $P$
    (with vertices in $\mathbb Z^2$)
    containing $b$ lattice points $\partial P\cap \mathbb Z^2$
    in its boundary and $i$ lattice points in its interior.}:

  \begin{lem}\label{lemlatticebasis} Two linearly independent
    elements $u,v$ of a
    $2$-dimensional lattice $\Lambda$ 
    form a basis of the lattice $\Lambda$
    if and only if the triangle with vertices $(0,0),u,v$
    contains no other elements of $\Lambda$.
\end{lem}

\begin{proof} This is an easy corollary of Pick's Theorem.

  It follows also from the observation that the parallelogram with vertices
  $(0,0),u,v,u+v$ is a fundamental domain of the sub-lattice $\mathbb Zu+\mathbb Z v$ of $\Lambda$ spanned by $u$ and $v$.
\end{proof}

\begin{lem}\label{lemnoncrossing} If $f_1,f_2$ and $g_1,g_2$ are two bases of a $2$-dimensional 
  lattice $\Lambda=\mathbb Z f_1+\mathbb Z f_2=\mathbb Z g_1+\mathbb Z g_2$
  such that $\{\pm f_1,\pm f_2\}$ and $\{\pm g_1,\pm g_2\}$ do not intersect,
  then $\{\pm g_1,\pm g_2\}$ is contained
  in a two opposite connected components
  of $\mathbb R^2\setminus(\mathbb R f_1\cup \mathbb R f_2)$.
\end{lem}

Lemma \ref{lemnoncrossing} can be remembered easily:
The lines $\mathbb Rf_1,\mathbb R f_2$ and $\mathbb R g_1,\mathbb R g_2$
defined by two generating sets $f_1,f_2$ and $g_1,g_2$ of a two-dimensional
lattice are never intertwined.

\begin{proof} Up to sign-changes and up to exchanging the roles of $f_1$ and
  $f_2$ we have otherwise $f_1=\alpha g_1+\beta g_2$
  and $f_2=\gamma g_1-\delta g_2$ where $\alpha,\beta,\gamma,\delta$ are
  strictly positive integers. This implies that $g_1$ belongs to the
  segment joining $\frac{1}{\alpha}f_1$ to $\frac{1}{\gamma}f_2$
contained in the convex hull of
  $(0,0),f_1,f_2$. The assumption $g_1\not\in\{f_1,f_2\}$ shows that
  this contradicts Lemma \ref{lemlatticebasis}.
\end{proof}

We consider the eight open cones
  of $\mathbb R^2$ forming the complement of the four lines defined by
  $xy(x^2-y^2)=0$. We call these eight open cones \emph{windmill-cones}
  (mainly as an attempt to turn the content of this paper into a piece of
  loftier mathematics)
  and we colour them alternately black
  and white, starting with a black E-NE windmill-cone
  $\{(x,y)\ \vert\ 0<y<x\}$ (using the conventions
  of wind-roses).

  A sub-lattice $M$ of finite index in $\mathbb Z^2$ has
  a \emph{black (respectively white) monochromatic basis} if
  it is generated by two elements 
  $b_1,b_2$ such that the set $\{\pm b_1,\pm b_2\}$ intersects all four
  open black (respectively white) windmill-cones. 
  
\begin{lem}\label{lemonecolour} All monochromatic bases of a lattice have
  the same colour.
\end{lem}

\begin{proof} If $b_1,b_2$, respectively $w_1,w_2$, is a black, respectively
  white, monochromatic basis of
  $\mathbb Z b_1+\mathbb Zb_2=\mathbb Z w_1+\mathbb Z w_2$ then the four
  lines defined by $\mathbb R b_1,\mathbb R b_2$ and
  $\mathbb R w_1,\mathbb R w_2$
  are intertwined in contradiction with Lemma \ref{lemnoncrossing}.
\end{proof}

An odd prime-number $p$ and an integer $\mu$ define a sub-lattice 
\begin{align}\label{deflambdamu}
  \Lambda_\mu(p)=\{(x,y)\in\mathbb Z,\ \vert\ x+\mu y\equiv 0\pmod p\}
\end{align}
of index $p$ in $\mathbb Z^2$.

\begin{prop}\label{propexistencebasis} Every lattice $\Lambda_\mu(p)$
  with $2\leq \mu\leq p-2$ has a monochromatic basis.
\end{prop}

\begin{proof} $\Lambda_\mu(p)$ contains obviously no elements of the form
  $(\pm m,0)$ or $(\pm m,\pm p)$ with $m$ in $\{1,2,\ldots,p-1\}$.
  Since $p$ is prime, $\Lambda_\mu(p)$ contains no elements of
  the form $(0,\pm m),(\pm p,\pm m)$ with $m$ in $\{1,\ldots,p-1\}$.
  Moreover, for $\mu$ in $\{2,\ldots,p-2\}$ considered as a subset of the finite
  field $\mathbb Z/p\mathbb Z$, the elements $1+\mu$ and $1-\mu$ are invertible in $\mathbb Z/p\mathbb Z$. This implies that $\Lambda_\mu(p)$ contains also
  no elements of the form $(\pm m,\pm m)$ with $m$ in $\{1,\ldots,p-1\}$.
  The intersection of a (black or white) windmill-cone with $[-p,p]^2$
  defines thus a triangle of area $p^2/2$ whose boundary contains no
  lattice-points of $\Lambda_\mu(p)$ except for its three vertices.
  Lemma \ref{lemlatticebasis} implies now that
  every open (black or white) windmill-cone contains a non-zero element $(x,y)$
  of $\Lambda_\mu(p)$
  with coordinates $x,y$ in $\{\pm 1,\pm 2,\ldots,\pm(p-1)\}$. Let $b_1$ be such a point
  in the black E-NE windmill-cone and let $b_2$ be such a point in the
  black N-NW windmill-cone. Let $\mathcal Q$ be the parallelogram
  with vertices $\pm b_1,\pm b_2$. If the interior of $\mathcal Q$
  contains a non-zero element $\tilde b$
  of $\Lambda_\mu(p)$ in a black windmill-cone,
  replacing $b_1$ or $b_2$ by $\pm \tilde b$ yields 
  a smaller parallelogram $\mathcal Q'$ strictly contained in $\mathcal Q$.
  Iterating this construction leads
  finally to a parallelogram $\tilde Q$ such that
  $\tilde Q\cap \Lambda_\mu(p)$
  intersects open black windmill-cones only in vertices.
  If $\tilde Q$ contains no elements of $\Lambda_\mu(p)$ in
  open white windmill-cones
  we get a black monochromatic basis by Lemma \ref{lemlatticebasis}
  Otherwise $\pm b_1,\pm b_2$ generate a strict sub-lattice of
  $\Lambda_\mu(p)$ and Lemma \ref{lemlatticebasis}
  implies that $\tilde Q\cap \Lambda_\mu(p)$
  intersects all four white windmill-cones in
  non-zero elements $\pm w_1,\pm w_2$ of $\Lambda_\mu(p)$. Switching colours and
  restarting the previous construction with the parallelogram
  spanned by $\pm w_1,\pm w_2$ ends the proof.
  \end{proof}

We call a black monochromatic basis $u,v$ of a lattice $\Lambda_\mu(p)$
(with $\mu$ in $\{2,\ldots,p-2\}$) \emph{reduced} if $u=(a,c),v=(-d,b)$
with $a,b,c,d\in\mathbb N$ such that $\min(a,b)>\max(c,d)$.

\begin{lem}\label{lemreddefdbyone} A reduced black monochromatic basis is uniquely
  defined by one of its elements.
\end{lem}

\begin{proof} Let $u=(a,c),v=(-d,b)$ be a reduced black monochromatic basis
  of $\Lambda=\mathbb Z u+\mathbb Z v$. The element $v$ belongs
  necessarily to one of the two closest affine lines parallel to $\mathbb Ru$
  which intersect $\Lambda$. Since $v$ belongs to the open black N-NW
  windmill-cone, $v$ belongs to the closest line $L_+$
  intersecting $\Lambda$
  which is parallel to $\mathbb R u$ and strictly above $\mathbb R u$.
  Reducedness of the basis $u,v$ shows that $v$ is the rightmost element
  of the intersection of $L_+\cap \Lambda$ with the open black
  N-NW windmill-cone.

  An analogous argument shows that $v$
  determines $u$ uniquely.
\end{proof}

\begin{prop}\label{propmonochr} Given an odd prime number $p$, a lattice $\Lambda_\mu(p)$ 
with $\mu$ in $\{2,\ldots,p-2\}$ has either only white monochromatic bases
or it has a unique reduced black monochromatic basis.
\end{prop}

\begin{proof} Proposition \ref{propexistencebasis} shows that
  such a lattice contains either black or white monochromatic bases.
  They are either all black or all white
  by Lemma \ref{lemonecolour}.
  We can thus assume that $\Lambda_\mu(p)$ has only black monochromatic
  bases. We chose such a basis with $u$ and $v$ respectively in the open
  black E-NE and N-NW windmill-cone. Replacing $u$ if necessary with $u-sv$
  we can assume that $u$ is the lowest element of the open E-NE windmill
  cone which belongs to $u+\mathbb Rv\cap \Lambda_\mu(p)$. Replacing
  similarly $v$ with $v+tu$ we can similarly assume that $v$ is
  the rightmost element of the open N-NW windmill-cone which belongs
  to $v+\mathbb Ru\cap \Lambda_\mu(p)$.

  The set $u=(a,c),v=(-d,b)$ is clearly still a black monochromatic basis of
  $\Lambda_\mu(p)$. We claim that $u,v$ is reduced. Observe first
  that the inclusion of $u$ in the open black E-NE windmill-cone implies
  $0<c<a$. The inclusion of $v$ in the open black N-NW windmill-cone
  implies similarly $0<d<b$.
  Since $\Lambda_\mu(p)\cap \mathbb R(1,0)=\mathbb Z(p,0)$, we have
  either $u-v=(p,0)$ which leads to the contradiction
  $\Lambda_\mu(p)=\mathbb Z(p,0)+\mathbb Z(1,0)$ or the vector $u-v=(a+d,c-b)$
  belongs to the lower half-plane and we have $b>c$.
  An analogous argument shows that $v+u$ belongs to the
  half-plane $\{(x,y),x>0\}$. This implies the inequality $a>d$.
  We have thus a black
  basis $u=(a,c),v=(-d,b)$ with $a,b,c,d$ in $\mathbb N$ such that $\min(a,b)>\max(b,c)$. This shows that $u,v$ is a reduced black monochromatic basis.

  Assume now that $\Lambda_\mu(p)$ has two reduced black
  monochromatic bases
  $u=(a,c),v=(-d,b)$ and $u'=(a',c'),v'=(-d',b')$.
  Lemma \ref{lemreddefdbyone} shows that $\mathbb Ru,\mathbb R v$ and
  $\mathbb Ru',\mathbb Rv'$ are four distinct lines which are not
  intertwined by Lemma \ref{lemnoncrossing}. Up to exchanging $u,v$
  with $u',v'$, we can suppose that 
  $u',v'$ belong to the open cone $(0,+\infty)u+
  (0,+\infty)v$ spanned by $u$ and $v$. We have thus $u'=\alpha u+\beta v$ and $v'=\gamma u+\delta v$
  with $\alpha,\beta,\gamma,\delta$ strictly positive integers. Reducedness
  of the black monochromatic basis $u,v$ implies that $v+u$ does not belong
  to the black N-NW windmill-cone containing $v$. It is thus either an element
  of the closure of the white N-NE windmill-cone
  or it belongs to the black E-NE windmill-cone containing $u$.

  Suppose first that
  $u+v$ is an element of the closed white N-NE windmill-cone.
  Since $\alpha+\beta$ and $\gamma+\delta$ are both at least
  equal to $2$ and since $u',v'$ belong respectively to the
  black E-NE and the N-NW windmill cones, the element
  $u+v$ belongs to the closed segment joining
  $\frac{2}{\alpha+\beta}u'$ to $\frac{2}{\gamma+\delta}v'$
  contained in the convex hull of $(0,0),u',v'$.
  This is a contradiction by Lemma
  \ref{lemlatticebasis}.

  All lattice points $v+u,v+2u,v+3u,\ldots$ are  
  thus elements of the open black E-NE windmill-cone containing $u$.
  The affine line $L=\gamma u+\mathbb R v$ intersects thus 
  $\Lambda_\mu(p)$ in
  at least two elements $\gamma u,\gamma u+v$ of the E-NE windmill-cone.
  Since $v$ has slope strictly smaller than $-1$, the intersection
  of $L=\gamma u+
  \mathbb R v$ with the white N-NE windmill-cone is strictly longer
  than the intersection of $L$ with the black E-NE windmill-cone.
  The intersection of $L$ with the open white E-NE windmill-cone contains
  thus at least one element of $\Lambda_\mu(p)$.
  This implies $\delta\geq 3$ where
  $$v'=(-d',b')=\gamma u+\delta v=\gamma(a,c)+\delta(-d,b)$$
  and we get $b'=\gamma c+\delta b\geq c+3b$. 

  Since the
  open strip bounded by the two parallel lines $\mathbb R v$ and
  $(a,c)+\mathbb R v$ contains no elements
  of $\Lambda_\mu(p)$ and since $v$ has slope strictly smaller than
  $-1$, an element $(x,y)$
  of $\Lambda_\mu(p)$ contained in the open black E-NE windmill-cone
  satisfies $x\geq a/2$. Applying this to $u'=(a',c')$ we get $a'\geq a/2$.
  We have now
  $$p=a'b'+c'd'>a'b'\geq \frac{a}{2}(c+3b)=ab+a\frac{b+c}{2}>ab+ac>ab+cd=p$$
  ending the proof.
\end{proof}


\begin{proof}[Proof of Theorem \ref{thmmain}]
  Given an odd prime number $p$, we denote by
  $\mathcal S_p$ be the set of all associated
  solutions $(a,b,c,d)$
  defined by Theorem \ref{thmmain}.

  We associate to a solution $(a,b,c,d)$ in $\mathcal S_p$
  the two vectors $u=(a,c),\ v=(-d,b)$
and we consider the sub-lattice
$\Lambda=\mathbb Zu+\mathbb Zv$ of index $p=ab-c(-d)$ in $\mathbb Z^2$ generated by $u$ and $v$.
Since $p$ is prime, there are exactly two solutions with $cd=0$, given by
$(p,1,0,0)$ and $(1,p,0,0)$ corresponding to the
lattices $\mathbb Z(p,0)+\mathbb Z(0,1)$ and
$\mathbb Z(1,0)+\mathbb Z(0,p)$.

We suppose henceforth $cd>0$.
The vectors $u$ and $v$ are contained respectively in
the black E-NE and the
black N-NW windmill-cone and form a reduced black
monochromatic basis of the lattice $\Lambda$.

Sub-lattices of prime-index $p$ in $\mathbb Z^2$
are in bijection with the set of all $p+1$ points on
the projective line $\mathbb P^2\mathbb F_p$ over the finite
field $\mathbb F_p$. More precisely, a point $[a:b]$
of the projective line defines the
lattice
$$\Lambda_{[a:b]}=\{(x,y)\in\mathbb Z^2\ \vert\ ax+by\equiv 0\pmod p\}$$
corresponding to the lattice $\Lambda_\mu$ (defined by
(\ref{deflambdamu})) with $\mu\equiv b/a
\pmod p$ for $a$ invertible.

We have already considered lattices associated to the two solutions
with $cd=0$. The lattices corresponding to
$\mu\equiv\pm 1\pmod p$
have no monochromatic basis and yield thus no solutions.
All $(p-3)$ lattices $\Lambda_\mu$ with $\mu\in\{2,\ldots,p-2\}$
have monochromatic bases by Proposition \ref{propexistencebasis}.

Since $\Lambda_\mu$ and $\Lambda_{p-\mu}$
  (respectively $\Lambda_{\mu^{-1}\pmod p}$) differ by a horizontal
(respectively diagonal) reflection,
they have monochromatic bases of different colours.
Proposition \ref{propmonochr} shows thus that there are
$(p-3)/2$ different values of $\mu$ in $\{2,\ldots,p-2\}$
which give rise to a lattice $\Lambda_\mu$ corresponding to
a solution in $(a,b,c,d)$ in $\mathcal S_p$ with $cd\not=0$.
The set $\mathcal S_p$ contains thus exactly
$(p-3)/2+2=(p+1)/2$ elements.
\end{proof}

\begin{rem} The lattice $\Lambda=\mathbb Z(a,c)+\mathbb Z(-d,b)$
  associated to a solution $(a,b,c,d)$ in $\mathcal S_p$
  has a fundamental domain given by the union
  of the rectangle of size $a\times b$ with vertices $(0,0),
  (a,0),(a,c),(0,c)$ and of the rectangle of size $d\times c$
  with vertices $(a,0),(a+d,0),(a+d,c),(a,c)$.
\end{rem}

  \section{Complements}

  \subsection{Constructing the solution associated to $\pm \mu$
    in $\{2,\ldots,p-2\}$}\label{ssectionsolution}

  Every pair of opposite elements $\pm \mu$ represented by
  an integer $\mu\in\{2,\ldots,p-2\}$ defines exactly
  one solution in $\mathcal S_p$ and all solutions except
  $(p,1,0,0)$ and $(1,p,0,0)$ are of this form.
  The associated solution can be constructed as follows:
  Gau\ss ian lattice-reduction applied to
  $$\Lambda_\mu(p)=\mathbb Z(p,0)+\mathbb Z(-\mu,1)=
  \{(x,y)\ \vert\ x+\mu y\equiv0\pmod p\}$$
  yields a basis containing
  a shortest vector $w$ in $\Lambda_\mu(p)$. Proposition
  \ref{propshortvector} below shows how to deduce from
  this a black monochromatic basis either of $\Lambda_\mu(p)$ or
  of $\Lambda_{-\mu}(p)$. The first part of the proof of Proposition
  \ref{propmonochr} shows how to construct a reduced basis $(a,c),(-d,b)$ (associated to the solution $(a,b,c,d)$ defined by
  $\Lambda_{\pm \mu}(p)$) from a black monochromatic
  basis.
  
  \begin{prop}\label{propshortvector} Given an odd prime number $p$ and $\mu$ in
    $\{2,\ldots,p-2\}$, let $w$ be a shortest non-zero element of
    $\Lambda_\mu(p)$. There exists a monochromatic basis of
    $\Lambda_\mu(p)$ which contains either $w$ or a shortest element
    of $\Lambda_\mu(p)\setminus\mathbb Z w$.
  \end{prop}

  \begin{proof} After a rotation by a suitable angle
    $k\pi/2$ and perhaps a horizontal
    reflection, we end up with a lattice $\Lambda$ having a shortest
    non-zero element $w$ in the open black E-NE windmill-cone. 
    Let $L_+$ be the closest affine line above $\mathbb Rw$ which is
    parallel to $\mathbb R w$ and intersects $\Lambda\setminus\mathbb Z w$.
    If the intersection of $L_+$ with the open black N-NW windmill-cone
    contains an element $r$ of $\Lambda$, we get a black monochromatic
    basis by considering $w,r$.

    Otherwise an easy geometric argument shows that $L_+$ intersects $\Lambda$ in a rightmost point $v$ of the
    open white W-NW windmill-cone and in a leftmost point $u$ of the
    open white N-NE windmill-cone and we get a white monochromatic basis
    by considering $u,v$. Since $u,v$ are separated by the black N-NW
    windmill-cone containing the orthogonal line to $\mathbb R w$,
    either $u$ or $v$ is a shortest element of $\Lambda\setminus \mathbb Z w$.
  \end{proof}
  
  \subsection{Projective statistics}

  We say that a function $f:\mathbb N^4\setminus \{0,0,0,0)\}
  \longrightarrow \mathbb R$
  defines a \emph{projective statistic} if $f(\lambda a,\lambda b,\lambda c,\lambda d)=f(a,b,c,d)$ for all $\lambda\geq 1$
  (i.e. if $f$ factorises through the projection of $\mathbb N^4$
  into $\mathbb P^3\mathbb R$). Interesting examples when
  studying the sets $\mathcal S_p$ of solutions to Theorem
  \ref{thmmain} are perhaps
  $$\frac{c+d}{a+b},\frac{cd}{ab},\frac{c+d}{\sqrt{a b}},
  \frac{\min(c,d)}{\max(a,b)},\frac{\min(a,b)}{\max(a,b)},\
  \mathrm{ etc.}.$$
  We assume henceforth $f$ continuous on (an open set of) $\mathbb P^3\mathbb R$
  and we are interested
  in the asymptotic (with respect to $p\rightarrow \infty)$)
  proportion $\mu_f(\Omega)$
  of elements in $\mathcal S_p$
  given by the preimage $f^{-1}(\Omega)\subset \mathcal S_p$
  of an open set $\Omega$ in $\mathbb R$. (Equivalently, one can
  also consider the asymptotic proportion of all elements in
  $\mathcal S_p$ projecting on an open set $\mathcal O$ of
  $\mathbb P^3\mathbb R$. Observe however that the closure of 
  $\cup_p\mathcal S_p$ avoids most of $\mathbb P^3\mathbb R$.)

  If $\mu_f(\Omega)$ exists (which should be the case for
  all reasonable continuous projective statistics $f$)
  the probability measure $\mu_f(\Omega)$ is perhaps
  equal to an integral explained below.

  The value $\mu_f(\Omega)$ can of course also be approximated
  almost surely, either by
  computing the set $f^{-1}(\Omega)\subset \mathcal S_p$
  for a large prime $p$, or by choosing a large number of
  pairs $(p_i,\mu_i)$
  with $p_i$ large primes and $\mu_i$ chosen uniformly
  among $\{2,\ldots,p-2\}$ leading to a lattice $\Lambda_{\pm \mu_i}(p_i)$ having a reduced black basis and estimating $\mu_f(\Omega)$
  as the proportion of choices which lead to solutions $(a,b,c,d)$
  (computed using
  for example Section \ref{ssectionsolution})
  in $\mathcal S_{p_i}$ with $f(a,b,c,d)$ in $\Omega$.

  We are now going to explain a computation
  of $\mu_f(\Omega)$ which is exact under an assumption of equidstribution.
  The famous modular curve $\mathcal M=\mathbb H/\mathrm{PSL}_2(\mathbb Z)$
  is the moduli space for rank 2 lattices in $\mathbb C$, up
  to orientation-preserving similitudes. An obvious quotient $\tilde U$
  of the unitary tangent bundle of $\mathcal M$ corresponds to
  sub-lattices of $\mathbb C$ up to positive real scalings
  (or equivalently to geodesics of the orbifold
  $\mathcal M$ containing a marked point: Given a sublattice $\Lambda$ of
  $\mathbb C$ with shortest non-zero vector $w$ consider the corresponding
  point on the standard fundamental domain for
  $\mathcal M$ together with the geodesic having the slope
  of $w$ at this point).
  It can thus be identified with the set of all sub-lattices of $\mathbb C$ with a given determinant. $\mathcal U$ has a natural finite
  probability measure $\nu_{\mathcal U}$.
  We denote by $\mathcal U_b$, respectively
  by $\mathcal U_w$ the subset of all elements of $\mathcal U$
  corresponding to sub-lattices of $\mathbb C$ having a black, respectively white, monochromatic basis. The complement $\mathcal U\setminus(\mathcal U_b\cup\mathcal U_w)$ is of measure zero and can
  be neglected. The function $f$ defines now a continuous function
  $\tilde f$ on an open subset of $\mathcal U_b$.
  Assuming asymptotical
  equidistribution in $\mathcal U$ with respect to $\nu_{\mathcal U}$ of
  sub-lattices of prime index $p$ in $\mathbb Z^2$, we have
  $\mu_f(\Omega)=\nu_{\mathcal U}(\tilde f^{-1}(\Omega))/\nu_{\mathcal U}(\mathcal U_B)$ (where $\nu_{\mathcal U}(\mathcal U_b)=1/2$
  if $\nu_{\mathcal U}$ is a probability measure with total measure
  $1$ on $\mathcal U$). This reduces the computation of
  $\mu_f(\Omega)$ to an integration (of a complicated function
  on a complicated subset of $\mathcal U$).
  
  \subsection{Variations}

  One can of course also consider the equation $n=ab+cd$ for arbitrary
  $n$. There are two possibilities when $n=ab$: either require $c=d=0$
  or accept solutions with $cd=0$ but $c+d\in\{0,1,\ldots,\min(a,b)-1\}$.
  Both problems can be
  solved by the techniques of this paper, up to technicalities.

  Also interesting is the equation $n=ab-cd$ with $\min(a,b)>\max(c,d)$.
  The number of solutions in $\mathbb N^4$ can be shown to be
  $$\sum_{d\vert n,\ d^2\geq n}\left(d+1+\frac{n}{d}\right)\ .$$
  Details will hopefully be provided ulteriorly in another paper.

\noindent Roland BACHER, 

\noindent Univ. Grenoble Alpes, Institut Fourier, 

\noindent F-38000 Grenoble, France.
\vskip0.5cm
\noindent e-mail: Roland.Bacher@univ-grenoble-alpes.fr

\end{document}